\documentclass{amsart}

\title{$R(3,10) \leq 41$}
\author{Vigleik Angeltveit}
\address{Mathematical Sciences Institute \\
Australian National University \\
Canberra, ACT 0200 \\
Australia}

\usepackage{amsxtra}
\usepackage{amsfonts}
\usepackage[latin1]{inputenc}
\usepackage{graphicx}
\usepackage{amsmath,amssymb,latexsym,amsthm,mathrsfs}
\usepackage[all]{xy}
\usepackage{pstricks}
\usepackage{verbatim}

\newtheorem{theorem}{Theorem}[section]

\newtheorem{lemma}[theorem]{Lemma}

\newtheorem{cor}[theorem]{Corollary}

\newtheorem{prop}[theorem]{Proposition}

\theoremstyle{definition}

\makeatletter
\let\c@equation\c@theorem
\makeatother
\numberwithin{equation}{section}

 \newcommand{\cR}{\mathcal{R}}        

\begin{document}

\begin{abstract}
We improve the upper bound on the Ramsey number $R(3,10)$ from $42$ to $41$. Hence $R(3,10)$ is equal to $40$ or $41$.
\end{abstract}

\maketitle

\section{Introduction}
The Ramsey number $R(s,t)$ is defined to be the smallest $n$ such that every graph of order $n$ contains either a clique of $s$ vertices or an independent set of $t$ vertices. See \cite{Ra94} for a survey on the currently known bounds for small Ramsey numbers. The smallest Ramsey numbers that are currently unknown are $R(4,6)$ and $R(3,10)$.

\begin{theorem} \label{t:main}
The Ramsey number $R(3,10)$ is less than or equal to $41$.
\end{theorem}

In \cite[Theorem 1]{Ex89b}, Exoo proved that $R(3,10) \geq 40$ by explicitly constructing an $R(3,10,39)$-graph. This means that $R(3,10)$ is either $40$ or $41$. The proof of Theorem \ref{t:main} uses extensive computer calculations, and can be thought of as a follow-up to \cite{GoRa13} where Goedgebeur and Radziszowski proved that $R(3,10) \leq 42$.

The project described in \cite{GoRa13} took a total of about 50 CPU years, and lowering the upper bound on $R(3,10)$ from 42 to 41 requires several orders of magnitude more calculations. For example, Goedgebeur and Radziszowski had to consider about 80 million $R(3,8)$-graphs and we had to consider approximately 150 billion such graphs. Despite this, we were able to complete the proof of Theorem \ref{t:main} in about 3 CPU years, and all of the calculations were done on two standard desktop computers over a period of a few months.

To complete this project, the main challenge was to come up with algorithms that are several orders of magnitude faster than those used in \cite{GoRa13} to go with the $4$ orders of magnitude more graphs.

\subsection{Acknowledgements}
The author would like to thank Brendan McKay for doing some sanity checks on some of the sets of graphs generated in this project and for commenting on an early draft. In particular, he spent some CPU time verifying that the census of $\cR(3,9,32,e \leq 112)$ graphs is complete. I would also like to thank Jan Goedgebeur for some useful comments and for sharing a file with $43$ $117$ $868$ $R(3,10,39)$-graphs he and Radziszowski found in \cite{GoRa13}.

This project also used the software package \emph{nauty} extensively. See \cite{McPi14}.

\section{Outline of proof}
Let $\cR(s,t,n)$ denote the set of isomorphism classes of Ramsey graphs of type $(s,t)$ with $n$ vertices. Similarly, let $\cR(s,t,n,e \leq e_0)$ denote the set of such graphs with at most $e_0$ edges. Given a hypothetical graph $\Gamma$ in $\cR(3,10,41)$, the basic idea is to consider a vertex $v$ of degree $d$. Then the neighbourhood of $v$ is an independent set of size $d$, while the dual neighbourhood is a graph $G \in R(3,9,41-d-1)$. We can then use a version of the Neighbourhood Gluing Extension Method from \cite{GoRa13} to reconstruct $\Gamma$, and it suffices to show that this kind of gluing does not produce any output.

\begin{prop} \label{p:R39graphs}
Any graph $\Gamma$ in $\cR(3,10,41)$ must have a vertex $v$ with dual neighbourhood in one of the following sets:
\begin{enumerate}
 \item $\cR(3,9,32,e \leq 112)$
 \item $\cR(3,9,33,e \leq 121)$
 \item $\cR(3,9,34,e \leq 130)$
\end{enumerate}
\end{prop}

\begin{proof}
First, $\cR(3,9,35,e \leq 139)$ should be on that list but by \cite[Table 14]{GoRa13} the set of such graphs is empty.
Since $R(3,9) = 36$, every vertex of $\Gamma$ must have degree in $\{5,6,7,8,9\}$. Because 41 is odd, $\Gamma$ cannot be regular of degree 9 and must have at least one vertex of degree in $\{5,6,7,8\}$.

This result makes intuitive sense, since if $\Gamma$ has $40$ vertices of degree $9$ and a single vertex $v$ of degree $8$ then the dual neighbourhood of $v$ is in $\cR(3,9,32,e \leq 112)$. If we start removing edges from $\Gamma$, the dual neighbourhood of $v$ will still be in one of these sets.

Let $n_i$ denote the number of vertices of degree $i$ in $\Gamma$.

First suppose $n_5=n_6=n_7=0$, so every vertex of $\Gamma$ has degree $8$ or $9$. Then $\Gamma$ has $\frac{369-n_8}{2}$ edges, and the dual neighbourhood of a degree $8$ vertex $v$ has $\frac{369-n_8}{2} - \sum\limits_{i=1}^8 d(w_i)$ edges, where the sum is over the $8$ vertices in $\Gamma$ adjacent to $v$. Because the induced subgraph of $\Gamma$ on the degree $8$ vertices is triangle-free and $n_8$ is odd, one of the degree $8$ vertices is adjacent to at most $\frac{n_8-1}{2}$ other degree $8$ vertices. Hence the sum of the degree of the neighbours of $v$ is at least $72 - \frac{n_8-1}{2}$.

It follows that the number of edges in the dual neighbourhood $\Gamma_v^{-}$ of that degree $8$ vertex $v$ is
\[
 e(\Gamma_v^{-}) = \frac{369-n_8}{2} - \sum\limits_{i=1}^8 d(w_i) \leq \frac{369-n_8}{2} - (72-\frac{n_8-1}{2}) = 112.
\]
Hence we have established that if every vertex of $\Gamma$ has degree $8$ or $9$ then the dual neighbourhood of some vertex of $\Gamma$ is in $\cR(3,9,32,e \leq 112)$.

For the general case, first note that $\Gamma$ has $\frac{369-N}{2}$ edges, where $N = n_8 + 2n_7 + 3n_6 + 4n_5$. For a vertex $v$ of degree at most $8$, let
\begin{multline*}
 \epsilon(v) = |\textnormal{degree 8 neighbours of $v$}| + 2|\textnormal{degree 7 neighbours of $v$}| \\
 + 3|\textnormal{degree 6 neighbours of $v$}| + 4|\textnormal{degree 5 neighbours of $v$}|.
\end{multline*}
Then the dual neighbourhood of $v$ is in one of the sets in the lemma if and only if $\epsilon(v) \leq \frac{N-1}{2}$.

Again we use that the induced subgraph of $\Gamma$ on the degree at most $8$ vertices is triangle-free: Consider the graph $\Gamma'$ with $9-i$ vertices $\{v^j\}_{j=1}^{9-i}$ for each vertex $v$ of $\Gamma$ of degree $i$ for $i \in \{5,6,7,8\}$, with $v^j$ and $w^k$ adjacent whenever $v$ and $w$ are adjacent in $\Gamma$ (and no edges between $v^j$ and $v^k$ for fixed $v$). Then $\Gamma'$ has $N$ vertices and is also triangle-free. It follows that some vertex in $\Gamma'$ has degree at most $\frac{N-1}{2}$. The degree of $v^j$ in $\Gamma'$ is equal to $\epsilon(v)$, and the result follows.
\end{proof}

\begin{cor}
To prove Theorem \ref{t:main}, it suffices to glue the graphs in Proposition \ref{p:R39graphs} using the Neighbourhood Gluing Extension Method.
\end{cor}

This leaves us with two things to do:
\begin{enumerate}
 \item Generate all the graphs in Proposition \ref{p:R39graphs}.
 \item Show that none of them extend to an $\cR(3,10,41)$-graph.
\end{enumerate}

\section{A census of graphs}
To generate the graphs in Proposition \ref{p:R39graphs}, we start with a partial census of $\cR(3,7)$ and work our way up.

\subsection{A partial census of $\cR(3,7)$}
We used the following sets of graphs:
\begin{enumerate}
 \item $\cR(3,7,16,e \leq 24)$ (46 514 graphs)
 \item $\cR(3,7,17,e \leq 30)$ (3 131 580 graphs)
 \item $\cR(3,7,18,e \leq 36)$ (23 149 358 graphs)
 \item $\cR(3,7,19,e \leq 41)$ (2 173 527 graphs)
 \item $\cR(3,7,20,e \leq 46)$ (10 613 graphs)
\end{enumerate}

These are not that difficult to compute, and we computed them using a one-point extender that takes as input an $R(3,7,n)$-graph and outputs all ways to add a single vertex to produce an $R(3,7,n+1)$-graph. It is possible to get away with only using $\cR(3,7,16,e \leq 21)$ rather than $\cR(3,7,16,e \leq 24)$, but see Section \ref{s:R3824e63} for a place where we found it convenient to use the larger set of graphs.

\subsection{A partial census of $\cR(3,8)$}
We need the following sets of graphs:
\begin{enumerate}
 \item $\cR(3,8,23,e \leq 53)$ (238 854 716 graphs)
 \item $\cR(3,8,24,e \leq 63)$ (approximately 150 billion graphs)
 \item $\cR(3,8,25,e \leq 70)$ (2 120 846 970 graphs)
 \item $\cR(3,8,26,e \leq 77)$ (1 767 543 graphs)
 \item $\cR(3,8,27)$ \hspace{29pt} (477 142 graphs)
\end{enumerate}

The graphs in $\cR(3,8,27)$ are not logically necessary, but we found it convenient to glue them as well in order to make an additional assumption about the minimum degree of any vertex for the remaining gluing operations.

In each case it is a small linear programming exercise to show that gluing the above $R(3,7)$-graphs suffices to generate these $R(3,8)$-graphs.

We remark that we do not have an accurate count of the graphs in $\cR(3,8,24,e=63)$ as we did not store them, but we have $|\cR(3,8,24,e \leq 62)| = $ 14 645 288 701. See Section \ref{s:R3824e63} below for details on how to deal with $\cR(3,8,24,e = 63)$.

\subsection{A partial census of $\cR(3,9)$}
As explained in Proposition \ref{p:R39graphs} above, we need the following sets of graphs
\begin{enumerate}
 \item $\cR(3,9,32,e \leq 112)$ (1 554 928 360 graphs)
 \item $\cR(3,9,33,e \leq 121)$ (14 395 graphs)
 \item $\cR(3,9,34,e \leq 130)$ (5 graphs)
%  \item $\cR(3,9,35)$ (1 graph)
\end{enumerate}

In addition we considered the single graph in $\cR(3,9,35)$. This is not logically necessary, but we found it convenient to glue it in order to make an additional assumption about the minimum degree of any vertex for the remaining gluing operations.

Again, in each case it is a small linear programming exercise to show that gluing the above $R(3,8)$-graphs suffices to generate these $R(3,9)$-graphs.

We remark that $\cR(3,8,24,e=63)$ is only needed to generate the regular degree $7$ graphs in $\cR(3,9,32,e=112)$. For all the remaining graphs it suffices to consider $\cR(3,8,24,e \leq 62)$.

\section{Gluing algorithms}
The main algorithm used by Goedgebeur and Radziszowski is the Neighbourhood Gluing Extension Method. Similar algorithms are pervasive in the subject. We made a few modifications to increase the performance.

\subsection{The Neighbourhood Gluing Extension Method}
The basic algorithm is as follows. Given $G' \in \cR(3,t,n)$ and some natural number $d$, one can compute all graphs $G \in \cR(3,t+1,n+1+d,e \leq e_0)$ with a vertex $v$ of degree $d$ with neighbours $v_1,\ldots,v_d$ and dual neighbourhood $G'$ as follows. First, make a list of all independent sets in $G'$. If $S$ is an independent set in $G'$ then assigning $S \cup \{v\}$ as the set of neighbours of some $v_i$ does not create any triangles, and the only thing to check is if assigning $S_1,\ldots,S_d$ (and $v$) as the set of neighbours to $v_1,\ldots,v_d$ creates any independent $(t+1)$-sets. This is a valid assignment if and only if for each subset $K \subseteq \{1,\ldots,n\}$ the independence number of $VG' \setminus \bigcup\limits_{k \in K} S_k$ is at most $t-|K|$. (This is authomatic for $|K|=1$.) If so, we call $S_1,\ldots,S_d$ \emph{compatible}. Based on this one can implement an inductive search algorithm. By ordering the independent sets by size one can also prune when it is clear that adding independent sets from later in the list will produce a graph with too many edges.

We used their algorithm without change to compute the parts of $\cR(3,8,n)$ we needed, but we made some changes when computing $\cR(3,9,n)$ and when showing that $\cR(3,10,41)$ is empty.

\subsection{Our first modification: Using only maximal independent sets}
This modification is responsible for the largest performance increase. To describe it, we use the following result:

\begin{lemma}
If the independent sets $S_1,\ldots,S_d$ are compatible and $S_i \subseteq S_i'$ for each $i$ then the independent sets $S_1',\ldots,S_d'$ are also compatible.
\end{lemma}

In other words, if we can add an edge from $v_i$ to $G'$ without introducing any triangles then we can add that same edge to $G$ without introducing any triangles. Adding an edge does not introduce any additional independent sets, so the lemma follows. Hence it suffices to consider independent sets that are \emph{maximal} in the sense that they are not contained in any larger independent sets. This has a massive advantage. In a typical case this reduces the number of independent sets by an order of magnitude, and the number of tuples by several orders of magnitude.

This modification also has some disadvantages. The first is that we can no longer prune by number of edges, at least not at this stage of the algorithm.

The second downside is that such a maximal solution might represent a large number of non-maximal solutions. So we need another algorithm to extract those non-maximal solutions.

We do this as follows: Before searching for compatible $d$-tuples of maximal independent sets, assign each independent set $S$ to a maximal independent set $S'$ with $S \subset S'$. This way each maximal independent set comes with an allowed list of subsets.

Given compatible maximal independent sets $S_1',\ldots,S_d'$, we inductively try replacing $S_i'$ by each allowed subset $S_i$ while pruning on the total number of edges.

\subsection{Our second modification: Not using a precomputed table of independence numbers}
Goedgebeur and Radziszowski used a precomputed table of the independence number of any subset of $VG'$. Precomputing this simply takes too long. If the number of vertices is small (at most 27) we initialise a table of independence numbers to $0$ and compute independence numbers as needed. If the number of vertices is large, we use a HashMap instead.

Note that there is a fast way to check if a pair of independent sets are compatible without using a precomputed table of independence numbers: For each independent set $S$ we precompute which independent $(t-1)$-sets are contained in $VG' \setminus S$ and store the result in a bitvector. Then we can use a bitwise \textit{and} operation to check if $S_i$ and $S_j$ are compatible. We then store this result for each pair in another bitvector, so we can check with another bitwise \textit{and} which independent sets are pairwise compatible with $S_1,\ldots,S_k$.

\subsection{Our third modification: Computing independence numbers as late as possible}
Suppose, for example, that we start with $G' \in \cR(3,8,24)$ and $d = 7$, and that we are looking for graphs in $\cR(3,9,32,e \leq 112)$. Then we are looking for compatible independent $7$-tuples $S_1,\ldots,S_7$ of maximal independent sets. Because bitwise operations are fast and computing the independence number of some subset of $VG'$ is slow, we check for pairwise compatibility only (which we compute once and then store in a bitvector) until we find a potential solution $S_1,\ldots,S_7$. Only at this point do we start computing independence numbers, starting with that of $VG' \setminus (S_1 \cup S_2 \cup S_3)$. If we do find that $\{S_1,S_2,S_3\}$ are incompatible, we can jump straight to the next assignment of $S_3$.

\subsection{Our fourth modification: Ordering the independent sets}
We order the maximal independent sets $S_1,\ldots,S_N$ as follows: First, we find the independent $(t-1)$-set which is contained in $VG' \setminus S_i$ for the largest number of maximal independent sets, and consider those maximal independent sets last. The reason for doing so is that none of them are compatible with each other, so we can choose at most one of them and they only need to be considered in the inner-most loop when choosing $S_d$. Hence we can work with a smaller collection of independent sets until the inner-most loop. This can be repeated.

When aiming for $\cR(3,10,41)$ this gives us an especially large saving, as we almost never have to consider the last 3 or 4 subsets of maximal independent sets.

\subsection{Performance}
For example, these modifications allow us to glue approximately 250 graphs in $\cR(3,8,24)$ per second (per CPU core) on a modern computer to produce graphs in $\cR(3,9,32,e \leq 112)$, and a similar number of graphs in $\cR(3,9,32)$ per second to produce (hypothetical) graphs in $\cR(3,10,41)$.

\section{The special case of extending $\cR(3,8,24,e=63)$ to $\cR(3,9,32,e=112)$} \label{s:R3824e63}
There are approximately 150 billion graphs in $\cR(3,8,24,e=63)$, and these are needed to find the regular degree $7$ graphs in $\cR(3,9,32,e=112)$ only. With the obvious algorithm this would take much longer than any of the other calculations, but we can take advantage of the fact that we only need them to find the regular degree $7$ graphs to produce a much more efficient algorithm.

If $\Gamma \in \cR(3,9,32,e=112)$ is regular of degree $7$ then \emph{every} vertex of $\Gamma$ has dual neighbourhood in $\cR(3,8,24,e=63)$.

Suppose in addition that the dual neighbourhood of some vertex $v$ of $\Gamma$ in $\cR(3,8,24,e=63)$ has a degree $7$ vertex $w$. That means the neighbourhoods of $v$ and $w$ in $\Gamma$ are disjoint, and the intersection of the dual neighbourhoods of $v$ and $w$ is a graph $G' \in \cR(3,7,16)$. Hence $\Gamma$ is given by gluing two graphs $G_1, G_2 \in \cR(3,8,63,e=112)$ along $G' \in \cR(3,7,16)$.

For a fixed $G' \in \cR(3,7,16)$ we first find all $G \in \cR(3,8,24,e=63)$ with $G'$ as a dual neighbourhood using the usual Neighbourhood Gluing Extension Method. For each pair $G_1, G_2 \in \cR(3,8,24,e=63)$ intersecting in $G'$ we can then compute the degree of each vertex $x \in VG'$ considered in $\Gamma$, and each must have degree exactly $7$. That is a very strong condition, so very few pairs are compatible. For a fixed $G'$ there might be several hundred million extensions to $\cR(3,8,24,e=63)$, so that leaves approximately $10^{17}$ pairs in that case. To avoid having to check every pair, we put the graphs in a HashMap based on the degree sequence of the vertices in $VG'$ considered in $G$. Then, for each $G_1 \in \cR(3,8,24,e=63)$ extending $G'$ we can simply look up the compatible $G_2$ in the HashMap.

After finding a pair that is compatible, we do the following: For each neighbour $v_i$ of $v$ and neighbour $w_j$ of $w$, we add an edge between $v_i$ and $w_j$ if this does not introduce a triangle. This can be checked with a single bitwise \textit{and} for each pair $(v_i,w_j)$. If every vertex now has degree at least 7, we check if this is an $R(3,9,32)$-graph. If it is, we then remove edges in all possible ways between degree $\geq 8$ vertices and record the $R(3,9,32)$-graphs we get.

A similar algorithm works for $G' \in \cR(3,7,17)$ or $G' \in \cR(3,7,18)$. The restriction on the degree of the vertices in $G'$ considered in $\Gamma$ is weaker, but we also have fewer pairs.

Because this has to work for \emph{every neighbourhood}, we can exclude any graph $G \in \cR(3,8,24,e=63)$ that has already been considered. It is possible to do this without storing $\cR(3,8,24,e=63)$, for example as follows:

We first ran this program for $G' \in \cR(3,7,16,e=24)$. Next, we ran it for $G' \in \cR(3,7,17,e=30)$ while excluding any output in $\cR(3,8,24,e=63)$ with some dual neighbourhood in $\cR(3,7,16,e=24)$. Next we considered $\cR(3,7,18,e=36)$ while excluding any output in $\cR(3,7,16,e=24)$ or $\cR(3,7,17,e=30)$. After that we considered $\cR(3,7,16,e=23)$ while excluding the graphs with a dual neighbourhood in one of the previous cases, and so on until covering all of $\cR(3,7,16,e \leq 24)$, $\cR(3,7,17,e \leq 30)$ and $\cR(3,7,18,e \leq 36)$. Finally we also considered $\cR(3,7,19,e \leq 40)$ (which is small and only required a simple one-point extender.)

This produced a total of $506$ regular degree $7$ graphs in $\cR(3,9,32,e=112)$, in about 6 months of CPU time.

\section{Extending the necessary $\cR(3,9)$-graphs}
In the end we had about 1.6 billion graphs in $\cR(3,9,32,e \leq 112)$, 14395 graphs in $\cR(3,9,33,e \leq 121)$ and 5 graphs in $\cR(3,9,34,e \leq 130)$. The latter two cases were very fast. The first case took about 3 months of CPU time to extend to $\cR(3,10,41)$. Since this produced no outputs, this finishes the proof of Theorem \ref{t:main}.

Since running a program that produces no output is not very satisfying we had our program produce a large number of graphs is $\cR(3,10,38)$ instead, and used a one-point extender to verify that none of them extended to $\cR(3,10,41)$ (or even to $\cR(3,10,40)$).

\section{A partial census of $\cR(3,10,39)$}
This project produced a large number of graphs in $\cR(3,10,39)$, and we were able to produce many more by forgetting a single vertex and then using a one-point extender in all possible ways. We have produced 39 745 077 such graphs, and we can still produce more pretty easily. Most of these had already been found by Goedgebeur and Radziszowski \cite{GoRa13} who found 43 117 868 such graphs. By taking the union (and forgetting a single vertex and using a one-point extender) we extended this to 43 146 537 graphs. In contrast, the number of maximal Ramsey graphs in other known cases is much smaller:

\begin{table}[h]
  \begin{center}
    \caption{The number of maximal Ramsey graphs in some cases}
    \label{tab:table1}
    \begin{tabular}{l|c|r} % <-- Alignments: 1st column left, 2nd middle and 3rd right, with vertical lines in between
      $|\cR(3,5, 13)|$ & 1 \\
      \hline
      $|\cR(3,6,17)|$ & 7 \\
      \hline
      $|\cR(3,7,22)|$ & 191 \\
      \hline
      $|\cR(3,8,27)|$ & 477 142 \\
      \hline
      $|\cR(3,9,35)|$ & 1 \\
      \hline
      $|\cR(4,4,17)|$ & 1 \\
      \hline
      $|\cR(4,5,24)|$ & 352 366 \\
      \hline
      $|\cR(4,6,35)|$ & $\geq$ 37 \\
      \hline
      $|\cR(5,5,42)|$ & $\geq$ 656 \\
      \hline
      $|\cR(3,10,39)|$ & $\geq$ 43 146 537
    \end{tabular}
  \end{center}
\end{table}

The obvious conjecture is that $R(3,10) = 40$. But because $\cR(3,10,39)$ is so large, we are not confident in this prediction. If, indeed, $R(3,10) = 40$ then this is going to be quite difficult to prove with current techniques. For example, a potential $R(3,10,40)$-graph might be regular of degree $9$ and then the dual neighbourhood of any vertex will be in $\cR(3,9,30,e=99)$. This contains many orders of magnitude more graphs than $\cR(3,9,32,e \leq 112)$.

% \bibliographystyle{plain}
% \bibliography{b.bib}

\end{document}